\documentclass[11pt,reqno]{amsart}
  
\usepackage{amssymb}
\usepackage{amsthm}
\usepackage{xspace}
\usepackage{amsmath}
\usepackage{setspace}
\usepackage{amscd}
\usepackage{tikz}
\usepackage{setspace}
\usepackage{enumerate}
\usepackage{graphicx}
\makeatletter
\@namedef{subjclassname@2010}{%
  \textup{2010} Mathematics Subject Classification}
\makeatother

\newtheorem{theorem}{Theorem}[section]
\newtheorem{lemma}[theorem]{Lemma}
\newtheorem{proposition}[theorem]{Proposition}
\theoremstyle{definition}
\newtheorem{definition}[theorem]{Definition}

\theoremstyle{remark}

\numberwithin{equation}{section}
 
\newcommand{\K}{\mbox{$\mathbb{K}$}}

\newcommand{\Z}{\mbox{$\mathbb{Z}$}}

\newcommand{\Q}{\mbox{$\mathbb{Q}$}}
\newcommand{\F}{\mbox{$\mathbb{F}$}}

\begin{document}

\title{ Prime polynomial values of \\ quadratic functions in short intervals}
\author{Sushma Palimar}
\address{
Department of Mathematics,\\ 
Indian Institute of  Science,\\
Bangalore, Karnataka, India. }

 \email{psushma@iisc.ac.in, sushmapalimar@gmail.com.}
\subjclass[2010]{11T55(primary),11P55,11N37.}
\begin{abstract}{In this paper we establish  the function field analogue
of Bateman-Horn conjecture in short interval  in the limit of a large finite field. Hence
 we start with counting prime polynomials generated by primitive
quadratic functions in short intervals. 
 To this end we further
 work out on function  field  analogs  of  cancellation of     Mobius sums and  its correlations(Chowla type sums) 
 and confirm that square root cancellation in Mobius sums is equivalent to square root cancellation in Chowla type  sums.
} 
\end{abstract}
\maketitle

\section{Introduction}
The well known conjecture of  Hardy-Littlewood  and Bateman-Horn predicts how often polynomials 
take prime values. For example, choose $f_{1}(T)$,..., $f_{r}(T)$  to be non-associate  irreducible polynomials in $\Z[T]$,
with leading coefficient of each
$f_{i}>0$ and suppose that for each prime  $p$ there exists $ n\in \Z$ such that 
$p\nmid f_{1}(n)\cdot\cdot\cdot f_{r}(n)$ for all integers $n$. 
Set $\pi_{f_1,f_2,...,f_r}$ as the number of positive integers $n\leq x$ such that 
$ f_{1}(n),...,f_{r}(n) \text{ are all primes}$. \begin{equation}\label{eqNo.1}
 \pi_{f_1,f_2,...f_r}(x):= \#\{1\leq n\leq x: f_{1}(n),...,f_{r}(n) \text{ are all primes} \}\end{equation}
 \[ \sim\frac{C (f_{1},f_{2},...,f_{r})}{   \prod\limits_{i=1}^{r} deg f_i}\frac{x}{(log x)^{r} } \]
where\[ C (f_{1},f_{2},...,f_{r}):=\underset{p \text{ prime}}{\bf{\prod}}\frac{1-\nu(p)/p}{({1-1/p})^{m}}, \]
$\nu(p)$ being the number of  solutions to $f_{1}(T)...f_{r}(T)\equiv {0}\pmod{p}$ in $\Z/p.$
 The product,  $C (f_{1},f_{2},...,f_{r}) $ is called Hardy-Littlewood constant associated to ${f_{1}(n),...,f_{r}(n)}$ \cite{kc}. 
 The only proved case of Bateman-Horn conjecture is the case of a single linear polynomial, which is the Dirichlet's theorem on
 primes in arithmetic progressions \cite{sb}.  Bateman Horn conjecture reduces  to the special case,
 Hardy-Littlewood  twin primes conjecture on the density of twin primes, whenever $r=2, f_1(T)=T; f_2(T)=T+2$, in (\ref{eqNo.1}) 
 according to which the number of twin primes pairs less than $x$ is:
 \[\pi_{2}(x)\sim 2\underset{p\geq 3}\Pi \frac{p(p-2)}{(p-1)^{2}}\frac{x}{(log x)^{2}}
 \]
 We derive the function field analog of Bateman-Horn conjecture  in the limit of large finite field in short interval. 
 \subsection*{Polynomial Ring and Prime Polynomials}
Let $\F_{q}[t]$ be the ring of polynomials over the finite field $\F_{q}$ with $q$ elements, $q=p^{\nu}, p: \text{ prime}.$
 Let  $\mathcal{P}_{n}=\{f\in \F_{q}[t]| \mathrm{deg} f=n \}$ be the set of all polynomials of degree $n$ and 
 $\mathcal{M}_{n}\subset\mathcal{P}_{n}$ be the subset of monic polynomials of degree $n$ over $\F_{q}$.
 The polynomial ring $\F_q[t]$ over a finite field $\F_q$ shares several properties with the ring of integers and the analogies between
 number field and function fields are fundamental in number theory. For instance, as quantitative aspect of this analogy,
 we have the Prime Polynomial Theorem.

 The prime polynomial theorem states that, the number $\pi_{q}(n)$ of monic irreducible polynomials of degree $n$ is \[
\pi_{q}(n)=\frac{q^{n}}{n}+ O\big(\frac{q^{n/2}}{n}\big), \quad q^{n}\rightarrow \infty.\]
The prime polynomial theorem for arithmetic progression  asserts,
given a polynomial modulus $Q\in \F_{q}[t]$, of positive degree and a polynomial $A$, coprime to $Q$, the number 
$ \pi_{q}(n; Q, A)$ of primes $P\equiv A\pmod Q , P\in \mathcal{M}_{n} $ satisfies,
\[ \pi_{q}(n; Q, A)=\frac{\pi_{q}(n)}{\Phi(Q)}+O(\mathrm{deg }Q. q^{n/2}).\]
where $\Phi_(Q)$ is the number of coprime residues modulo $Q$. For $q\rightarrow \infty,$ the
main term is dominant as long as $\mathrm{deg}Q<n/2.$
\subsection*{} In  \cite{lbs}
Bary-Soroker considered the function field analogue of  Hardy - Littlewood prime tuple conjecture,
in the limit of a large finite field, for functions, $F_i=f+h_i$,
 $h_i\in \mathbb{F}_q[t]$, $  deg(f)>deg(h_i), \text{ for } i=1,2,...,n$. This result was established   previously by
 Bender and Pollack \cite{AP} for the case $i=2$.
\subsection{Prime polynomials in short interval}
 Some of the salient problems of prime number theory deals with the study of distribution of primes 
 in short interval and arithmetic progression.
 To set up an equivalent problem for the polynomial ring $\F_{q}[t]$, we define short interval in function fields. Here we follow \cite{KZ} for notations.
 For a nonzero polynomial, $f\in \F_{q}[t]$, we define its norm by \(||f||:= {\#F_q[t]}/{(f)}=q^{deg f}\).
 Given a monic polynomial $f\in \mathcal{M}_{n}$ of degree $n$, and $h<n$, ``short intervals" around $f\in \mathcal{M}_{n}$ 
 of diameter $q^{h}$
 is the set
\begin{equation}I(f,h):=\{g\in\F_{q}[t]:\mathrm{deg}(f-g)\leq h\}=\{g\in \F_{q}[t]:||f-g||\leq q^{h}\} =f+ \mathcal{P}_{\leq h}\end{equation}
Thus, $I(f,h)$ is of the form,   $f+\sum_{i=0}^{h}a_it^i$, where ${\mathbf {a}}=(a_0,a_1,...,a_h)$ are
  algebraically independent variables over $\F_{q}$.
The number of polynomials in this interval is  \[H:=\#I(f;h)=q^{h+1}.\] 
For $h=n-1, I(f,n-1)=\mathcal{M}_{n}$ is the set of all polynomials of degree $n$.
For $h<n$, if $||f-g||\leq q^{h}$, then $f$ is monic if and only if $g$ is monic.
Bank, Bary-Soroker and Rosenzweig \cite{Baroli}
  obtained the  result on counting prime polynomials in the short  interval $I(A,h)$ 
  for the primitive linear function $f(t)+g(t)x$.
 In \cite{Baro}  the function field analogue of Hardy - Littlewood prime tuple conjecture on  these primitive linear functions is resolved 
 in short interval case.
  
  \subsection*{Counting Prime polynomials and HIT}
To establish the  function field analogue of  counting prime polynomials in \textit{short interval}  we start with irreducible quadratic function
$F(x,t)=f(t)+x^{2}\cdot g(t) \in \F_{q}[t][x]$ with following properties. Let
$f, g\in \F_{q}[t]$ be non zero, relatively prime polynomials,
$g(t)$ a monic polynomial and   the product 
$f\cdot g $  not a square polynomial with $\mathrm{deg}f<\mathrm{deg}g$. Hence, by the choice of $f\text{ and }g$, it is clear that, 
the function $F(x,t)= f(t)+x^2\cdot g(t)$ is irreducible in $x$. The first derivative of $F(x,t)$ is $2xg(t)\neq0$, implies
the function $ f(t)+g(t)x^2$, as a polynomial in $x$ is separable over $\F_{q}[t]$.
\subsection*{}
The short interval $I(p,m)$ defined as, $\mathrm{h}=p+\mathcal{P}_{\leq m}, \mathrm{deg}p>m,
\text{ is given by } $\begin{equation}\mathrm{h}(t)=p(t)+\sum_{i=0}^{m}a_it^i\end{equation} where ${\mathbf{a}}=(a_0,a_1,...,a_m)$ are
  algebraically independent variables over $\F_{q}$.
 ``Technically, the problem of   finding prime polynomials in short interval is 
  to find the number of tuples $\mathbf{A}=\{A_0,...,A_m\}\in \F_{q}^{m+1}$ 
 for which $F(\mathbf{A},t)$ is irreducible in $\F_{q}[t]$."
  The key tool used is the Hilbert Irreducibility Theorem, which answers, does 
the specialization $\mathbf{a}\mapsto\mathbf{A}\in\F_{q}^{m+1}$ preserve the irreducibility?
We have,\begin{equation}\label{maineq}
         F(x,t)=f(t)+x^2g(t) \text{ for } f(t),g(t)\in \F_{q}[t]
        \end{equation}
  Then, 
\[F(\mathrm{h},t)=f(t)+g(t)\mathrm{h}^2 = f(t)+g(t)\Big\{p(t)+\sum\limits_{i=0}^{m}a_it^i\Big\}^{2}\]
  
therefore,\begin{equation} \label{maineq1}F(\mathbf{a},t)= \tilde f(t)+g(t)\Big\{\big(\sum\limits_{j=0}^{m}a_jt^j\big)^2+2p(t)\sum\limits_{j=0}^{m}a_jt^j\Big\}\end{equation}
       where  \[\tilde f(t)=f(t)+g(t)p(t)^2 \]and\[ n=\mathrm{deg}F=\mathrm{deg}\tilde f=s+2k>\mathrm{deg}g\]                
 Under the above setup, we get an asymptotic for: 
 \begin{equation} \pi_{q}(I(p,m))=\#\{ h:=p(t)+\sum\limits_{j=0}^{m}A_jt^j| F(\mathbf{A},t) \text{ is irreducible in } \F_{q}[t]\}\end{equation}
and we have the following theorem.
\begin{theorem}\label{th1}
Let $n$ be a fixed positive integer and $q$ an odd prime power.
Then we have $\pi_{q}(I(p,m))=\frac{q^{m+1}}{n}+O_{n}(q^{m+\frac{1}{2}})$
\end{theorem}
One of the basic forms of HIT states
that, if $f(x_1,x_2,...,x_r,T_1,T_2,...,T_s)\in \Q[x_1,...,x_r,T_1,...,T_s]$  is irreducible, then there exists a specialization 
$\mathbf{(t)}=(t_1,...,t_s)$
such that $f(x_1,...,x_r)=f(x_1,...,x_r,t_1,...,t_s)$ as a rational polynomial in $x_1,...,x_r$ is irreducible over $\Q[x_1,...,x_r]$. If $r=1$,
  consider $f$ as a polynomial in $x$ over the rational function field
  $L=\Q(T_1,...,T_s)$, having roots $\alpha_1, . . . , \alpha_n$ in the algebraic closure $\bar L.$
  If  $f$ is irreducible and separable, then these roots are distinct, and we can consider the Galois group $G$ of $f$
  over $L$ as a subgroup of the symmetric group $S_n$. Then there exists a specialisation
 $ \mathbf{t}\in \Q^{s}$ such that the resulting rational polynomial in $x$ still is irreducible and has Galois group
 $G$ over $\Q$. In fact, if $\mathbf{t}$ is chosen in such a way that the specialized polynomial in $x$ 
 still is of degree $n$, and separable, then its Galois group $G_\mathbf{t}$ over $\Q$ is a subgroup of $G$
 (well-defined up to conjugation)  
  and it turns out that ‘almost all’ specializations for $\mathbf{t}$ preserve the Galois group, i.e. $G_\mathbf{t}=G$. Hence, we start by computing the 
  Galois group of $F(\mathbf{a},t)$ over $\bar \F_{q}(\mathbf{a})$. 
  In the sequel let $k=\bar \F_q$ and we prove $\mathrm{Gal}(F({\mathbf{a}},t),k({\mathbf{a}})) = S_{n}$  in section \ref{gal}.
\section{$\mathrm{Gal}(F({\mathbf{a}},t),k({\mathbf{a}})) = S_{n}$}
\label{gal}
 
\begin{theorem}\label{th2}
   Let  $k=\bar \F_{q},\, q \text{ an odd prime power and } {\mathbf{ a}}=(a_0,a_1,...,a_m)$ be $m+1$ tuple of varibles, 
 $m\geq 2$. 
 Then, $\mathrm{Gal}(F({\mathbf{a}},t),k({\mathbf{a}}))$ = $S_{n}$.
 \end{theorem}
To prove  Theorem \ref{th2}, as a first step we show the following:
\subsection{$\mathrm{Gal}(F({\mathbf{a}},t),k({\mathbf{a}}))$ is doubly transitive}      
\begin{proposition}\label{prop1}
 The polynomial function, $F(\mathbf{a},t)=\tilde f(t)+g(t)\{\sum\limits_{j=0}^{m}a_jt^j)^2+2p(t)\sum\limits_{j=0}^{m}a_jt^j\}$  
  is separable in $t$ and irreducible in the ring $k({\mathbf{a}})[t]$.
\end{proposition}

\begin{proof}
 To prove irreducibility of $F(\textbf{a},t)$ in $k({\textbf{a}})[t]$, 
 we consider $F({\textbf{a}},t)$ as a quadratic equation in the variable $a_0$
 and show that its  discriminant is square free. We have from equation (\ref{maineq1}), 
 \[F({\textbf{a}},t)=\tilde f(t)+g(t)\{(\sum\limits_{j=0}^{m}a_jt^j)^2+2p(t)\sum\limits_{j=0}^{m}a_jt^j\}\]
 \[F({\textbf{a}},t)=\tilde f(t)+g(t)\big\{l(t)^2+2l(t)p(t)\big\}\text{ where } l(t)=\sum\limits_{i=0}^{m}a_it^i.\]
 writing, 
 $ F(\mathbf{a},t)$  as a quadratic  equation in $a_0$, we have,
\[ g(t)a_0^2+\{2g(t)(l_1(t)+p(t))\}a_0+\big\{\tilde f(t)+g(t)\{l_1(t)^2+2p(t)l_1(t)\}\big \}=0\]
 \[\text{ where }l_1(t)=\sum_{i=1}^{m}a_it^i\]  
\text{ The discriminant of the above equation is } 
\begin{equation}\label{eqdisc}
\Delta(F({\mathbf{a}},t))= 
4g(t)^2\{l_1(t)+p(t)\}^2-4g(t)\{\tilde f(t)+g(t)\{l_1(t)^2+2p(t)l_1(t)\}\}\end{equation}
Substituting, $\tilde f(t)=f(t)+g(t)p(t)^2$, in the second sum of equation (\ref{eqdisc}), we have
\[4g(t)^2\{l_1(t)+p(t)\}^2 -4g(t)f(t)-4g(t)^2\{p(t)^2+l_1(t)^2+2p(t)l_1(t)\}\] Hence,
 
 \[\Delta(F({\mathbf{a}},t))=-4g(t)f(t)\neq 0\]
Clearly, $\Delta(F({\mathbf{a}},t)=-4f(t)g(t)$, is not a square by our choice of $f$ and  $g$. Therefore, $F(\mathbf{a},t)$ is irreducible 
in $k[a_1,...,a_m,t][a_0]=k[a_0,...,a_m,t]$. Hence, by Gauss lemma, $F({\mathbf{a}},t)$ is irreducible in $k(a_0,...,a_m)[t]$.
Coming to the separability of $F(\mathbf{a},t)$ in $t$, we see that,
  the irreducible polynomial  $F(x,t)$  is separable in $x$, (since the first derivative of $F(x,t)$ 
  is not a zero polynomial by the choice of $f$ and $g$.)
Hence,  the result by Rudnik in \cite{zr} confirms, the polynomial $F(\mathbf{a},t)$ is separable in $t$.
\end{proof}
In the next proposition we prove that the Galois group of  $F({\mathbf{a}},t)$ over $k(a_0,...,a_m)$ is doubly transitive with respect 
to the action on the roots of $F.$
We quickly go through the definitions of doubly transitivity as given in [page no.119, \cite{ssa}]. Let $K$ be a field.
Consider a polynomial $f(y)=y^n+a_1y^{n-1}+...+a_n$ with $a_i\in K.$ We can factor $f$ as $f(y)=(y-\alpha_1)(y-\alpha_2)...(y-\alpha_n)$,
where the roots $\alpha_i$ are in some extension field of $K.$ Let $L=K(\alpha_1,...,\alpha_n)$ then $L$ is called the splitting field of $K$.
The Galois group
of $L$ over $K$, denoted by $\mathrm{Gal}(L/K)$ is the group of all $K-$automorphisms of $L$.
i.e., those field automorphisms of $L$ which leave $K$, element wise fixed. Assuming $L$ to be separable over $K$, and $f$ to have no multiple factors
in $K[y]$, every member of $\mathrm{Gal}(L/K)$ permutes $\alpha_1,...,\alpha_n$, and this gives 
an injective homomorphism of $\mathrm{Gal}(L/K)$ into $S_{n}$ whose image is called Galois group of $f$ over $K$, denoted $\mathrm{Gal}(f,K)$.
$\mathrm{Gal}(f,K)$ is transitive if and only if $f$ is irreducible in $K[y] $ and 
$\mathrm{Gal}(f,K)$, a subgroup of symmetric group $S_n$
is $2-$ transitive if and only if $G$ is transitive and its one point stabilizer group $G_{\alpha_1}$ is transitive as a subgroup of $S_{n-1.}$
where, by definition, $G_{\alpha_1}=\{g\in G| g(\alpha_1)=\alpha_1\}$ is thought of as a 
subgroup of the group of all permutations of  roots $\{\alpha_2,...,\alpha_n\}$ of $f.$ Note that if $G$ is 
transitive, then all the one point stabilizers $G_{\alpha_i},i=1,2,...,n$  are isomorphic to each other.
To see the equational analogue of this, we consider, $f(y)$,  an irreducible polynomial in $K[y]$. We throw away a root of $f$, say $\alpha_1$
to get \[ f_1(y)=\frac{f(y)}{(y-\alpha_1)}=y^{n-1}+b_1y^{n-2}+...+b_{n-1}\in \K(\alpha_1)[Y].\] Then $f$ and $f_1$ are irreducible in $K[y]$ and 
$K(\alpha_1)[y]$ respectively if and only if $\mathrm{Gal}(f,K)$ is doubly transitive. \cite{ssa1}

\begin{proposition}\label{prop2}
 For, $F(\mathbf{a},t)$ defined above, the Galois group $\mathrm{G}$ of $F({\mathbf{a}},t)$ 
 over $k({\mathbf{a}})$ is doubly transitive with respect to the
 action on the roots of $F({\mathbf{a}},t)$.
\end{proposition}
\begin{proof}
Proposition   \ref{prop1} implies, the Galois group $\mathrm{G}=\mathrm{Gal}(F({\mathbf{a}},t),k({\mathbf{a}}))$
is transitive.
We show that,  the Galois group $\mathrm{Gal}(F({\mathbf{a}},t),k({\mathbf{a}}))$
is doubly transitive by specializing $a_0=0.$ 
Under the specialization $a_0=0$, we have
\[\tilde F(a_1,...,a_m,t)=\tilde f(t)+g(t)\{(\sum\limits_{j=1}^{m}a_jt^j)^2+2p(t)\sum\limits_{j=1}^{m}a_jt^j\}\]
 Let $\alpha \in k$ be a root of $\tilde f(t)$, by substituting, $t$ by $t+\alpha$, we may assume that, $\tilde f(0)=0.$ Hence, 
$ f_{0}(t)=\tilde f(t)/t$ is a polynomial.   
\[\tilde F(a_1,...,a_m,t)= t\big\{f_0(t)+g(t)\sum\limits_{i=1}^{m}\sum\limits_{j=1}^{m}a_ia_jt^{i+j-1}+g(t)2p(t)\sum\limits_{j=1}^{m}a_jt^{j-1}\big\}\]
We first show that,\begin{equation}\label{irr-dou}
                    f_0(t)+g(t)\sum\limits_{i=1}^{m}\sum\limits_{j=1}^{m}a_ia_jt^{i+j-1}+g(t)2p(t)\sum\limits_{j=1}^{m}a_jt^{j-1}
                   \end{equation}
is irreducible in $k(a_1,...,a_m)[t]$ and separable in $t$.

Separability of polynomial equation (\ref{irr-dou}) is attained   by applying the result in \cite{zr}.
Now, we show that,  equation(\ref{irr-dou}) is irreducible in $t$.
We prove this by writing,  
equation(\ref{irr-dou}) as a quadratic equation in $a_1$ and show that, discriminant of this quadratic equation is $f(t).g(t)$, 
which is  not a square polynomial.  Hence writing, equation(\ref{irr-dou})
  as a quadratic equation in $a_1$
we have
\begin{equation}\label{disc-dou}
\begin{split}
&t\cdot g(t)a_1^2+ g(t)\{\sum\limits_{j=2}^{m}a_jt^j+2p(t)\}a_1 +\\
&f_0(t)+g(t)\{\sum\limits_{i=2}^{m}\sum\limits_{i=2}^{m}a_ia_jt^{i+j-1}+2p(t)
\sum\limits_{i=2}^{m}a_it^{i-1}\}=0
\end{split}
\end{equation}
Discriminant of the above quadratic equation  (\ref{disc-dou})  in $a_1$ is
\[g(t)^{2}\{\sum\limits_{j=2}^{m}a_jt^j+2p(t)\}^2-4tg(t)\big[f_0(t)+g(t)\{\sum\limits_{i=2}^{m}\sum\limits_{i=2}^{m}a_ia_jt^{i+j-1}+2p(t)\sum\limits_{i=2}^{m}a_it^{i-1}\}\big]\]
substituting, $f_0(t)=\frac{1}{t}(f(t)+g(t)p(t)^2)$
\[g(t)^{2}\{\sum\limits_{j=2}^{m}a_jt^j+2p(t)\}^2-4f(t)g(t)-g(t)^2\{\sum\limits_{j=2}^{m}a_jt^j+2p(t)\}^2=-4f(t)g(t)\]
Clearly, equation (\ref{irr-dou}) is irreducible in $k(a_1,...,a_m)[t]$ and separable in $t$.
Let $\mathrm{G}_t$ be the Galois group of $\tilde F(a_1,...,a_m,t)$ over $k(\alpha,a_1,...,a_m)$. 
 Hence, from the discussion above  Proposition \ref{prop2} it is clear that, $\mathrm{G}_t$ is doubly transitive subgroup of the symmetric group
$S_{deg \tilde f}$.   
Since, $\tilde F(a_1,...,a_m,t)$ is separable, specialization induces, $\mathrm{G}_t\subset \mathrm{G}$, which is uniquely determined up to conjugation.
Hence, the stabilizer of a root of $F$ in $\mathrm{G}$ is transitive. Thus $\mathrm{G}$ is doubly transitive.
\end{proof}
\subsection*{ Proof of Theorem \ref{th2}}
\begin{proof}
 Already, we have seen Galois group of $F({\mathbf{a}},t)$ over $k({\mathbf{a}})$ is doubly transitive. Hence, it only remains to show
Galois group of $F({\mathbf{a}},t)$ over $k({\mathbf{a}})$ contains a tranposition.
To achieve this, we first show  that at some specialization $a_m=0$, the polynomial $F(a_0,...,a_{m-1})$ has one double zero and rest $(n-2)$ simple zeros.
$$\text{ Let }\tilde F(\mathbf{a},t)=F({\mathbf{a}},t)|a_m=0$$ 
\begin{definition}\label{Morse}
 A polynomial $f$ is called { \it Morse function }\cite{jp} if 
\begin{enumerate}
\item $f(\beta_i)\neq f(\beta_j),$ for $i\neq j$ i.e., critical values of $f$ are distinct.
 \item The zeros $\beta_1,\beta_2,...,\beta_{n-1}$ of derivative $f^{\prime}$ of $f$ are simple. i.e., 
  critical values of $f$ are non degenerate.
\end{enumerate}
\end{definition}
It is well known that,  discriminant of a monic separable polynomial
  is given by, \begin{equation}\label{disdef}disc(F)=\pm Res(F,F^{\prime})\end{equation}
  Proposition \ref{prop1}  implies, the specialized polynomial  $F({\mathbf{a}},t)|a_m=0$ (equation (\ref{speceq})) 
 is separable in $t$ and irreducible 
in $k(a_0,a_1,...,a_{m-1})[t]$.  We have,
 \begin{equation}\label{speceq}\tilde F(a_0,...,a_{m-1},t)=\tilde f(t)+g(t)\{(\sum\limits_{j=0}^{m-1}a_jt^j)^2+2p(t)\sum\limits_{j=0}^{m-1}a_jt^j\}\end{equation}
 Separability of $\tilde F$ implies for
$(A_0,A_1,...,A_{m-1}) \in \bar k^{m}$, the system of equations below does not have a solution in the algebraic closure of $k$.  
\begin{equation}\label{eq2}
 \begin{cases}
   \tilde F^{\prime}(\rho_i)=0\\
  \tilde F^{\prime}(\rho_j)=0\\
  \tilde F(\rho_i)=\tilde F(\rho_j)  \text{ for some} \rho_i,\rho_j \text{ in the algebraic closure of } {k}
  \end{cases}
\end{equation} 
 Which further imply, the critical values of $\tilde F(a_0,...,a_{m-1},t)$ are distinct. Proving condition $(1)$ of Definition {\ref{Morse}}.
  Detailed explanation is given in (\cite{CR}, Section 3 and Section 4).  
 It remains to prove, the condition $(2)$ of Definition \ref{disdef}, i.e., critical values of $\tilde F$ are non degenerate.
 A small calculation shows   derivative of $\tilde F(t)$ and derivative of  $\tilde F^{\prime}(t)$ with respect to $t$  have no common root. 
Thus, critical values of $\tilde F$ are non degenerate. 
Hence, the function $\tilde F(a_0,...,a_{m-1},t)$ is Morse. Hence, the polynomial $F(a_0,...,a_{m-1})$ has one double zero and rest $(n-2)$ simple zeros.
Hence a transposition in $G=\mathrm{Gal}(F({\mathbf{a}},t),k({\mathbf{a}}))$ is implied by (\cite{Ho}, Lemma1) which is stated below.
\begin{lemma}\label{lem2}
 Let $p$ be a prime number and $\mathfrak{p}$ be a prime ideal in $K$ satisfying $\mathfrak{p}|p.$
 If $f(x)\equiv (x-c)^{2}\bar h(x) \pmod{p}$ for some $c\in \Z$ and  a separable polynomial $\bar h(x)\equiv \pmod {p}$ 
 such that $\bar h(c)\not\equiv 0{\pmod {p}}$,
 then the inertia group of $\mathfrak{p}$ over $\Q$ is either trivial or a group generated by a transposition.
\end{lemma}
By Proposition \ref{prop2}, the  Galois group 
$G$ of $F({\mathbf{a}},t)$ over $k({\mathbf{a}})$ is doubly transitive.
Any finite doubly transitive permutation group containing a transposition is a full symmetric group (\cite{jp} Lemma, 4.4.3). 
Thus, $\mathrm{Gal}(F({\mathbf{a}},t), k({\mathbf{a}}))$ is isomorphic to the full symmetric group $S_{n}$.
Thus, Theorem \ref{th2} is complete.
\end{proof}
\section{Irreducibility Criteria}
Since, $\mathrm{Gal}(F({\mathbf{a}},t),k({\mathbf{a}}))= S_{n}$. Now, we obtain the asymptotic for the number of irreducibles in
$I(p,m)$ for which the specialized polynomial $F({\mathbf{A}},t)$ is irreducible in $\F_{q}[t]$, where $\mathbf{A}=(A_0,A_1,...,A_m)\in \F_q$.
To attain this, we  invoke an irreducibility criteria,  as in  [Lemma 2.8, \cite{lbmj}], which reduces the above problem of finding 
irreducibles $h\in I(p,m)$ to counting of rational points of an absolutely irreducible variety over a finite field $\F_{q}$.
Then the required asymptotic follows by applying $Lang-Weil$ estimate.
   Now, we have the following proposition.
\begin{proposition}\label{Theorem2}
 Let ${\mathbf{a}}=(a_0,a_1,...,a_m)$ be an $(m+1)$ tuple of variables. Let $F({\mathbf{a}},t)\in \F_{q}[a_0,a_1,...,a_m,t]$ 
 be a polynomial that is  separable in $t$ and irreducible in the ring $k(\mathbf{a})[t]$ with $\mathrm{deg}_{t}F=n$. Let $L$ be the splitting field 
of $F({\mathbf{a}},t)$ over $\F_{q}(\mathbf{a})$. Let $k$ be an algebraic closure of $\F_{q}$. Assume that, $\mathrm{Gal}(F,k(a_0,...,a_m))=S_n.$
Then  the number of $\mathbf{A}=(A_0,...,A_m) \in \F_{q}^{m+1} $
 for which all the specialized polynomial $F(\mathbf{A},t)$, is irreducible is
$\frac{q^{m+1}}{n}\big(1+O_{n}(q^{-1/2})\big)$ as $q\rightarrow\infty$ and $n$ is fixed.
\end{proposition}
\begin{proof}
This is proved in [Lemma 2.1, \cite{lbs}]. 
\end{proof}
\subsection*{Proof of Theorem \ref{th1}}
\begin{proof}
 Let $n$ be fixed positive integer and $q$ an odd prime power and 
 \[F(\mathbf{a},t)= \tilde f(t)+g(t)(\sum\limits_{j=0}^{m}a_jt^j)^2+2p(t)g(t)\sum\limits_{j=0}^{m}a_jt^j.\]
 We have seen,  $\mathrm{Gal}(F(\mathbf{a},t),k(\mathbf{a}))=S_n$ 
  and  $F(\mathbf{a},t)$ satisfies all assumptions of Proposition \ref{Theorem2}. 
  Thus the number of $\{(A_0,...,A_m)\}\in \F_{q}^{m+1}$ for which $ F(\mathbf{A},t)$ 
  irreducible in $\F_q[t]$ is \[ \frac{q^{m+1}}{n}+O_{n}(q^{m+\frac{1}{2}}).\] 
 This finishes the proof since this number equals $\pi_q(I(p,m))$. Hence,  \[\pi_q(I(p,m))=\frac{\#I(p,m)}{n}+O_{n}(q^{m+\frac{1}{2}}).\]
\end{proof}
\section{ Cycle structure, Factorization type, Galois groups and Conjugacy classes}
In  previous sections, we obtained an asymptotic for the number of prime polynomials in the interval $I(p,m)$ 
for the function $F(x,t)=f(t)+x^2g(t)\in \F_{q}[t][x]$. 
Here, we derive an equidistribution result (Theorem \ref{frob}) by the  function field version of  Chebotarev Density theorem.
 We know that, factorization over $\F_{q}[t]$ resemble cycles of permutations,  below we state  some known results, mainly from  \cite{BR} and \cite{anlz}. 
 By definition, $\mathcal{M}_{n}$ 
 is the  collection of monic polynomials
 of degree $n$ consists of $q^{n}$ elements. Partition $\tau$ of a positive integer $n$ is defined to be a sequence of non-increasing positive integers
 $(c_1,...,c_k)$ such that, $|\tau|:=c_1+\cdot\cdot\cdot+c_k$ and $|\tau|=n$.
 \begin{definition}
  Every monic polynomial $f\in \F_{q}[t]$ of some degree $n$ has a factorization 
$f=P_1\cdot \cdot \cdot P_k$ in to 
irreducible monic polynomials $P_1,...,P_k \in \F_{q}[t]$, which is unique up to rearrangement. Taking degrees
we obtain  a partition of $n$ given by, $\mathrm{deg} P_1 + \cdot \cdot \cdot  +\mathrm{deg} P_k$ of $\mathrm{deg}f$ and
 its factorization type  is  given by $$\tau_{f}=(\mathrm{deg}P_1,...,\mathrm{deg}P_k)$$ \end{definition}
 \begin{definition}
  Every permutation $\sigma \in S_n$ has a cycle decomposition $\sigma=(\sigma_1...\sigma_k)$ in to
disjoint cycles $\sigma_1,...,\sigma_k$ which is unique up to rearrangement and each fixed point of $\sigma$ has cycle length 1. If $|\sigma_i|$
is the length of cycle $\sigma_i$, we   obtain  a partition of $n$ given by
 $$\tau_{\sigma}=(|\sigma_1|, \cdot \cdot \cdot  ,|\sigma_k|)$$  and it's  cycle type to be given by 
$\tau_{\sigma}=(|\sigma_1|+ \cdot \cdot \cdot  +|\sigma_k|)$
 \end{definition}
 For each partition $\tau\vdash n$,
the probability of  a random permutation  on $n$ letters has  a cycle structure $\sigma$ is given by Cauchy's formula:
\begin{equation}
\mathbb{P}(\tau_{\sigma}=\tau)=\frac{\#\{\sigma \in S_n:\tau_{\sigma}=\tau\}}{\# S_n}=\prod\limits_{j=1}^{k}\frac{1}{j^{c_j}\cdot c_j!}\end{equation}
As $q\rightarrow \infty$, the distribution over $\mathcal{M}_{n}$ of factorization types tends to distribution of cycle types in $S_{n}$ \cite{anlz}.
\begin{proposition}
 For a partition $\tau\vdash n$,
 \[ \lim\limits_{q\rightarrow \infty} \mathbb{P}_{f\in \mathcal{M}_{n}}(\tau_{f}=\tau)=\mathbb{P}_{\sigma\in S_n}(\tau_{\sigma}=\tau)\text{ 
where, } \]$\mathbb{P}_{f\in \mathcal{M}_{n}}(\tau_{f}=\tau):=\frac{1}{q^{n}}\#\{f\in \mathcal{M}_{n}:\tau_{f}=\tau\}$.
\end{proposition}

We consider specializations $F(\mathbf{A},t)$ as described in previous sections, where $\mathbf{A}=\{A_0,A_1,...,A_m\}\in \F_{q}$.
For such an $\mathbf{A}$ denote by $\Theta(F(\mathbf{A},t)$, the conjugacy class in $S_n$ of permutations with 
cycle structure $(d_1,d_2,...d_r)$, this is the factorization class of $F(\mathbf{A},t)$.
  For a separable polynomial $f\in \F_{q}[t]$ of degree $n$, the Frobenius map $\mathrm{Fr}_{q}$ given by 
$(y\rightarrow y^{q})$ defines a permutation of the roots of $f$, 
which gives a well defined conjugacy class $\Theta(f)$ of the symmetric group $S_n.$ The  degrees of the prime factors of $f$ correspond to
the cycle lengths of $\Theta (f)$. In particular, $f$ is irreducible if and only if $\Theta(f)$  (conjugacy class of) is a full cycle.
It is known that for any fixed conjugacy class $C$ of $S_n$ the probability of  $\Theta(F(\mathbf{A},t))=C$ as $\mathbf{A}$ ranges over $\F_{q}^{m+1}$ 
is determined by the Galois group $G$ of the polynomial $F(\mathbf{A},t)$ over the field $\F_{q}(\mathbf{A})$ together with its standard action on the roots,
up to an error term of $O_{m, \mathrm{deg}F}(q^{-\frac{1}{2}}))$. Hence, we have
\begin{theorem}\label{frob}
 Let ${\mathbf{a}}=(a_0,a_1,...,a_m)$ be an $(m+1)$ tuple of variables over $\F_{q}$. 
 Let $k$ be an algebraic closure of $\F_{q}$. Let $F({\mathbf{a}},t)\in \F_{q}[a_0,a_1,...,a_m,t]$ 
 be a polynomial that is  separable in $t$ and irreducible in the ring $k(\mathbf{a})[t]$ with $\mathrm{deg}_{t}F=n$. Let $L$ be the splitting field 
of $F({\mathbf{a}},t)$ over $\F_{q}(\mathbf{a})$.  Assume that, $G=\mathrm{Gal}(F,k(a_0,...,a_m))=S_n.$
Then for every conjugacy class  $C$ in $S_n$
\[\#\{h\in I(p,m)| \Theta(F({\mathbf{A}},t))=C\}= \frac{|C|}{|G}q^{m+1}(1+O_{n}(q^{-\frac{1}{2}}))\]
\end{theorem}
\begin{proof}
 The  proof follows from  Theorem 3.1 of  \cite{anlz}.
 \end{proof}
A variant of application of Theorem 3.1 in \cite{anlz} is given in ( Theorem 2.2., \cite{AE}).

\section{Bateman-Horn conjecture}
The classical Bateman-Horn problem is described in the introduction. 
In \cite{AE}, Entin has established  an analogue of Bateman and Horn conjecture under the following set up. 
Let $F_1,...,F_m\in \F_{q}[t][x], \mathrm{deg}_{x}F_i>0$ be
non-associate, irreducible and separable over $\F_{q}(t), n$ a natural number. 
Let $a_0,a_1,...,a_n$ be free variables,
$\mathrm{f}=a_nt^n+...+a_0\in\F_{q}[\mathbf{a},t]$ and $N_i=\mathrm{deg}_{t}F_i(t,\mathrm{f})$.
Under the above assumptions,   below Theorem  is established.
  
 \begin{theorem}\label{entbat}
  Let $F_1,...,F_m\in \F_{q}[t][x]$, $\mathrm{deg}_{x}F_{i}=r_i>0$, be non associate irreducible polynomials which are separable over
  $\F_{q}(t)$ (i.e., $F_i\not\in \F_q[t][x^p]$) and monic in $x$. Let $n$ be a natural number satisfying $n\geq3$ and $n\geq sl F_i$
  for $1\leq i\leq m$. Denote $N_i=r_in$. Denote by $\mu_i$ the number of irrducible factors into which $F_i(t,x)$ splits over $\bar \F_q$.
  Then \[\#\{f\in \F_q[t], deg f=n|F_i(t,f)\in \F_{q}[t]\text{ is irreducible for $i=1,2,..,r$}\}\]
  \[=\big(\prod\limits_{i=1}^{m}\frac{\mu_i}{N_i}\big)q^{n+1}(1+O_{m,\mathrm{deg}F_i,n}(q^{\frac{-1}{2}}))\]
 \end{theorem}
We study the function field version of Bateman-Horn conjecture for polynomial functions defined in equation (\ref{maineq}) namely
$F_i=f_i+g_ix^2 \in \F_{q}[t][x]$.
We obtain the following result.
\begin{theorem}\label{bat}
 Let each of $F_1,...,F_r\in \F_{q}[t][x]$,  distinct primitive quadratic functions satisfy all conditions of Proposition \ref{Theorem2}.
  Then \[\#\{h:=p(t)+\sum\limits_{j=0}^{m}A_jt^j|F_i(\mathbf{A},t)\in \F_{q}[t]\text{ is irreducible for $i=1,2,...,r$} \}\]\[=
 \big(\prod\limits_{i=1}^{r}\frac{1}{n_i}\big)q^{m+1}(1+O_{n,r}(q^{\frac{-1}{2}}))\]
 \[=\big(\prod\limits_{i=1}^{r}\frac{1}{n_i}\big)\#I(p,m)+O_{n,r}(q^{m+{\frac{1}{2}}})\]
\end{theorem}
\begin{proof}
 The proof of this Theorem is completed, once we show that, the Galois group $G$ is the full permutation group $S_{n_{1}}\times...\times S_{n_{r}}$
acting on the roots of $F_{i}(\mathbf{a},t)$ over $k(\mathbf{a})$.
From Lemma \ref{lem2}, we see that,  $\mathrm{Gal}(F_i({\mathbf{a}},t), k({\mathbf{a}}))\cong S_{n_i}$. To show that,
the Galois group $G$ is the full permutation group $S_{n_{1}}\times...\times S_{n_{r}}$
we need to show the multiplicative independence of $disc_{t}F_i(t,h)$ modulo squares, i.e.,
$disc_{t}F_{i}(\mathbf{a},t)$ is linearly independent as elements of $k(\mathbf{a})^{\times}/k(\mathbf{a})^{\times{2}}$, nothing but
 $d_i, d_j$ for $i\neq j$ are non squares and are relatively 
prime in the ring $k(\mathbf{a},t)$.
Denoting $d_i=disc_{t}F_{i}(\mathbf{a},t)$, the discriminant of $F_{i}(\mathbf{a},t)$, is considered as a polynomial in $t$.
Discriminant of a
monic separable polynomial $f(t)$ is defined by the resultant of $f \text{ and } f^{\prime}$:
\[disc(f)=\pm Res(f,f^{\prime})=\pm\prod_{j=1}^{\nu}f(\tau_j), \text{ where }f^{\prime}=c\prod_{j=1}^{\nu}(t-\tau_j) \]
Since, $\mathrm{Gal}(F_i({\mathbf{a}},t), k({\mathbf{a}}))$ is a full symmetric group, $d_i$ is not a square in $k(\mathbf{a})$ for any $i$.
If $d_i,d_j$ are not relatively prime in $k({\mathbf{a}})$, then they have a
common root. Thus $d_i,d_j$ having a common root  gives the following system of equations: 
\begin{equation}
 \begin{cases}
   F^{\prime}(\rho_i)=0\\
  F^{\prime}(\rho_j)=0\\
  F(\rho_i)=F(\rho_j)  \text{ for some } \rho_i,\rho_j \in \bar{k}.
  \end{cases}
\end{equation}
But we have seen, this system does not have any solution for any $\rho_i,\rho_j$   in the algebraic closure of $k$ [page 3, \cite{CR}].
Hence, the Galois group $\mathrm{G}=S_{n_{1}}\times...\times S_{n_{r}}$.
Rest of the proof follows from [Theorem 3.1, \cite{Baro}].
 Thus \[\#\{h:=p(t)+\sum\limits_{j=0}^{m}A_jt^j |F_i(\mathbf{A},t)\in \F_{q}[t]\text{ is irreducible for $i=1,2,...,r$} \}\]
\[=
 \big(\prod\limits_{i=1}^{r}\frac{1}{n_i}\big)q^{m+1}(1+O_{n,r}(q^{\frac{-1}{2}}))\]

\end{proof}

\section{M$\ddot{\mathrm{o}}$bius sums  and Chowla's conjecture}
The Mertens function,  given by the partial sums of M$\ddot{\mathrm{o}}$bius function, $M(n):=\sum\limits_{k=1}^{n}\mu(k)$ is of great
importance in Number Theory, where $\mu(k)$ is the Mobius function. For example, the Prime Number Theorem is logically
equivalent to \begin{equation}\sum\limits_{k=1}^{n}\mu(k)=o(n)\end{equation} and 
\begin{equation}\sum\limits_{k=1}^{n}\frac{\mu(k)}{k}=o(1)\end{equation} 
the Riemann Hypothesis is equivalent to 
\begin{equation}M(n)=O(n^{\frac{1}{2}+\epsilon})\text{ for all }\epsilon>0.\end{equation}
Thus one can say $M(n)$ is said to demonstrate square root cancellation.
Keating and Rudnick have established the function field version of square root cancellation of Mobius sums in short intervals \cite{KZ}.
Carmon and Rudnick have resolved the function field version of the conjecture of Chowla on auto-corelation of Mobius function \cite{CR} 
over large finite field and proved the following result.
  For $r,n\geq 2$, distinct polynomials $\alpha_1,...,\alpha_r\in\F_{q}[X]$ of degree
smaller than $n, q$ odd, and $(\epsilon_1,...,\epsilon_r)\in\{1,2\}$ not all even, 
\begin{equation}
 \sum_{deg F=n}\mu(F+\alpha_1)^{\epsilon_1}\cdot\cdot\cdot\mu(F+\alpha_1)^{\epsilon_r}=O(rnq^{n-\frac{1}{2}})
\end{equation}
We show that, there is square root cancellation in Mobius sums, as well as in 
the auto-correlation type sums appearing in Chowla's conjecture for the function $F(x,t)=f(t)+g(t)x^2$  in the short interval of the form $I(p,m)$
in large finite field (equation \ref{maineq1}).
\newline
For polynomials over a finite field $\F_q$, the Mobius function of a nonzero polynomial $f\in \F_{q}[x]$  is defined
to be $\mu(F)=(-1)^{r}$ if $F=cP_1,...,P_r$
with $0\neq c\in \F_{q}$ and $P_1,...,P_r$ are distinct monic irreducible polynomials, and $\mu(F)\neq 0$ otherwise.
The analogue of the full  sum $M(n)$ is the sum over all monic polynomials $\mathcal{M}_n$ of given degree $n$, for which 
we have \[
         \sum_{f\in \mathcal{M}_n}\mu(f) =
         \begin{cases}
          1, & n=0\\
          -q &n=1\\
          0, &n\geq 2
         \end{cases}
        \]
 
 \begin{equation}
 \text{ Set } S_{\mu}(p;m):=\sum\limits_{h\in I(p,m)}\mu(f+gh^2)
 \end{equation}
  In the next Theorem, we demonstrate 
square root cancellation in Mobius sums is equivalent to square root cancellation in auto correlation of Mobius sums
 in the short interval  $I(p,m)$ in the larger finite field 
  limit $q\rightarrow \infty$ and  $\mathrm{deg}(p)$ fixed.
 \begin{theorem}\label{th-mobi-sum}
 \begin{enumerate}
  \item 
 Let $F(\mathbf{a},t)$ satisfy the conditions of Propisition \ref{prop1} and   $\mathrm{deg}F(\mathbf{a},t)=n$. Then  for $m\geq 1$  
 \begin{equation}\label{mobsum}
 \big|S_{\mu}(p,m)\big| \ll_{n}\frac{\#I(p,m)}{\sqrt{q}}\end{equation} and the implied constant depending only on $n=\mathrm{deg}(F)$.
 \item Let each of $F_i$ of $\mathrm{max}$ $\mathrm{ deg} n_i$ satisfy all conditions of
  Proposition \ref{prop1}. Then for $\epsilon_{i}\in\{1,2\}$,
  not all even,
  \begin{equation}\label{autocho}
\Big|\sum\limits_{h\in I(p,m)} \mu(F_1(\mathbf{a},t))^{\epsilon_1}...\mu(F_r(\mathbf{a},t))^{\epsilon_r}\Big|\ll_{r,\mathrm{deg}F_i}\frac{\#I(p,m)}{\sqrt{q}}
 \end{equation}
 \end{enumerate}

 \end{theorem}
 \begin{proof}
   Mobius function $\mu(F)$ can be computed in terms of the discriminant $\mathrm{disc}(F)$ of $F(x)$ as (see \cite{kc1})
\(\mu(F)=(-1)^{\mathrm{deg} F}\chi_{2}(\mathrm{disc}(F)), \text{ where } \chi_{2}\)
is the quadratic character on $\F_{q}. $ 
Then   equation (\ref{mobsum}) becomes,
\begin{equation}\label{disc-char}
S_{\mu}(p;m):=(-1)^{n}\sum\limits_{h\in I(p,m)}\chi_{2}(\mathrm{disc}(f(t)+g(t)h^2))\end{equation}
To solve the equation (\ref{disc-char}), we follow the method as in \cite{CR}.
Since, $\mathrm{disc}(F)$ is polynomial in the coefficients of $F$, equation  (\ref{disc-char}) is an ${m+1}$ dimensional character sum,
the equation is evaluated by 
bounding all but one variable.   
writing, 
\begin{equation}
\begin{split}
&F(\mathbf{a},t)=\tilde F(\mathbf{a},t)+b\\
      \text{Set }  &D_{F}(b):=disc(\tilde F(\mathbf{a},t)+b) 
      \end{split}
         \end{equation}Here, $b:=F(0)$ of $F(\mathbf{a},t)=f(t)+g(t)h^{2}$
which is a polynomial of degree $(n-1)$ in $b$. Therefore we have 
\begin{equation}
 S_{\mu}(p;m) \leq \sum\limits_{\mathbf{A}\in \F_{q}^{m}}\Big|\sum\limits_{b\in \F_q}\chi_{2}(D_{F}(b))\Big|
\end{equation}
By Weil's theorem (the Riemann Hypothesis for curves over a finite field), which implies that for a polynomial $P(t)\in \F_{q}[t]$ of positive degree,
which is not proportional to a square of another polynomial (see \cite{CR})
\begin{equation}\label{cal-mobi}
 \Big|\chi_{2}(P(t))\Big| \leq(\mathrm{deg}P-1)\sqrt{q},\quad  P(t)\neq cH^{2}(t)
\end{equation}
Proposition \ref{prop1} implies $D_{F}(b)$ is not a square.
Non vanishing of $\chi_{2}(D_{F}(b))$ follows from the fact that, $F(x,t)$ is separable in $x$.  
Hence, we have
 \begin{equation}
  S_{\mu}(p;m) \leq (n-2)q^{m+\frac{1}{2}}
 \end{equation}
The implied constant depends only on $n=\mathrm{deg}F(\mathbf{a},t)$.
 
Proof of equation (\ref{autocho}) is based on the similar techniques used in proving equation(\ref{mobsum}) of  Theorem \ref{th-mobi-sum}. 
  Therefore we have from equation (\ref{cal-mobi})
$$
  \sum\limits_{{\mathbf{A}\in\F_{q}^{m}}}\Big|\sum\limits_{b\in{\F_q }}\chi_2(D_{F_i}(b)^{\epsilon_1}\cdot\cdot\cdot D_{F_r}(b)^{\epsilon_r})\Big|\leq 
  \Big(2\sum\limits_{i=1}^{r}(n_i-1)-1\Big)q^{m+\frac{1}{2}}
$$
 which is clearly,
 $$\sum\limits_{{\mathbf{A}\in\F_{q}^{m}}}\Big|\sum\limits_{b\in{\F_q }}\chi_2(D_{F_i}(b)^{\epsilon_1}\cdot\cdot\cdot D_{F_r}(b)^{\epsilon_r})\Big|
 \ll_{deg{F_i},r}\frac{\#I(p,m)}{\sqrt{q}}$$
 Hence, we can conclude that  square root cancellation in Mobius sums is equivalent to square root cancellation in Chowla type  sums.
 \end{proof}
  
 \section{Acknowledgement}
 The author would like to thank Prof. Soumya Das, Department of Mathematics, IISc, Bangalore, for suggesting the problems 
 and for many helpful discussions and suggestion while writing the manuscript.
 The author also would like to thank MathStack Exchange and MathOverflow Community.
 This research work was supported by the Department of Science and Technology, Government of India,  under the
 Women Scientist Scheme [SR/WOS-A/PM-31-2016].

 \bibliographystyle{plain}

\begin{thebibliography}{}
\expandafter\ifx\csname natexlab\endcsname\relax\def\natexlab#1{#1}\fi
\expandafter\ifx\csname url\endcsname\relax
  \def\url#1{\texttt{#1}}\fi
\expandafter\ifx\csname urlprefix\endcsname\relax\def\urlprefix{URL }\fi
\providecommand{\selectlanguage}[1]{\relax}

\bibitem{ssa1}
\textsc{S.S.Abhyankar,}
\newblock\emph{ Galois Theory on the line in nonzero characteristic.}
\newblock{Bull. Amer. MAth. Soc.(N.S.), 27, 1, 1992. }


\bibitem{ssa}
\textsc{ S.S.Abhyankar,}
\newblock\emph{ Geometry and Galois Theory.}
\newblock{Advances in Algebra and Geometry:University of Hyderabad Conference, 2001}
\newblock{ C. Musili, Editor}

\bibitem{anlz}
\textsc{ J. Andrade, L. Bary-Soroker,  and Z Rudnick,}
\newblock\emph{ Shifted convolution and the Titchmarsh divisor problem over $\F_q[t]$}
\newblock{ Philosophical transactions. Series A, Mathematical, physical, and engineering sciences, {\textbf{373}}, 2014.}

\bibitem{sb}
\textsc{S. Baier,}
\newblock\emph{On the Bateman-Horn Conjecture.}
\newblock{J. Number Theory,{\textbf{ 96}}, 432-448, 2002.}


\bibitem{Baro}
\textsc{E. Bank, L. Bary-Soroker.}
\newblock \emph{Prime polynomial values  of linear functions in short intervals.}
\newblock{  J. Number Theory, \textbf{151}, 263-275, 2015. }


\bibitem{Baroli}
\textsc{E. Bank, L. Bary-Soroker and L. Rosenzweig.} 
\newblock \emph{Prime polynomial values in short intervals and in arithmetic progressions.}
\newblock{ Duke Math. J. 164 (2):277-295, 2015.}

  
 \bibitem{lbs}
\textsc{L. Bary-Soroker,}
\newblock\emph{ Hardy-Littlewood tuple conjecture over large finite fields.}
\newblock{Int. Math. Res. Not. IMRN,(2):568-575,2014}

\bibitem{lbs1}
\textsc{L. Bary-Soroker,}
\newblock\emph{ Irreducible values of polynomials.}
\newblock{Adv. Math.,\textbf{229}(2):854-874,2012.}

\bibitem{lbmj}
\textsc{L. Bary-Soroker and M. Jarden,}
\newblock\emph{On the Bateman-Horn conjecture about polynomial rings.}
\newblock{Munster Journal of Mathematics,2012.}

\bibitem{AP}
\textsc{A. O.Bender and P. Pollack,}
\newblock\emph{On quantitative analogues of the Goldbach and twin prime conjectures over $\F_{q}[t]$.}
\newblock{arXiv:0912.1702 [math.NT]}



\bibitem{CR}
\textsc{D. Carmon  Z. Rudnick,}
\newblock\emph{The auto correlation of the Mobius function and Chowla's conjecture for the rational function field.}
\newblock{Q.J.Math., 65(1):53-61, 2014.}


\bibitem{kc}
\textsc{K. Conrad,}
\newblock\emph{Hardy-Littlewood Constants.}
\newblock{Mathematical Properties of Sequences and Other Combinatorial Structures, Springer Science + Business Media LLC.}
\newblock{Edited by, Jong-Seon No, Hong-Yeop Song, Tor Helleseth, P.Vijay Kumar}


\bibitem{kc1}
\textsc{K. Conrad,} 
\newblock\emph{irreducible values of polynomials: a non-analogy.}
\newblock{Number Fields and Function Fields: Two Parallel Worlds}
\newblock{Progress in Mathematics}, 239, Birkhauser, Basel, 2005.


\bibitem{AE}
\textsc{A. Entin,}
\newblock \emph{On the Bateman-Horn Conjecture for polynomials over large finite fields.}
\newblock{Compositio Math. 152:2525-2544, 2016.}


\bibitem{KZ}
\textsc{J.P. Keating and Z. Rudnick,}
\newblock\emph{The variance of the number of prime polynomials in short intervals and in residue classes.}
\newblock{Int. Math. Res. Not. IMRN, page 30 pp., 2012.}


\bibitem{lw}
\textsc{S. Lang and A. Weil,}
\newblock\emph{Number of points of varieties in finite fields.}
\newblock{Amer. J. Math., 76:819-827,1954.}

\bibitem{Ho}
\textsc{H. Osada,}
\newblock\emph{The Galois group of the polynomial $X^{n}+aX^{l}+b$.}
\newblock{J. Number Theory, 25, 230-238, 1987.}



\bibitem{BR}
\textsc{B. Rodgers,}
\newblock\emph{ Arithmetic functions in short intervals and the symmetric group.}
\newblock{Algebr Number Theory, 12(5):1243-1279, 2018.}

\bibitem{zr}
\textsc{Z. Rudnick,}
\newblock\emph{ Square-free values of polynomials over the rational function field.}
\newblock{J. Number Theory, 135, 2014. }


\bibitem{JP}
\textsc{J. P. Serre,}
\newblock\emph{Topics in Galois Theory (Research Notes in Mathematics.)}
\newblock{A.K.Peters Ltd., 2nd ed., 2008.}


\bibitem{jp}
\textsc{J.P. Serre,}
\newblock\emph{Topics in Galois Theory, {\it Research Notes in Mathematics.}}
\newblock{A K Peters Ltd., 2 edition, 2008. }


\end{thebibliography}

\end{document}